\documentclass[12pt]{amsart}

\usepackage{amsmath}
\usepackage{amsthm}
\usepackage{amssymb}
\usepackage{amscd}
\usepackage{bbold}
\usepackage[pdftex]{graphicx}
\usepackage{enumerate}

\newtheorem{lem}{Lemma}[section]
\newtheorem{thm}[lem]{Theorem}
\newtheorem{prop}[lem]{Proposition}

\theoremstyle{definition}
\newtheorem{defn}[lem]{Definition}
\newtheorem{example}[lem]{Example}

\newtheorem{rem}[lem]{Remark}

  \title{Sparse Reconstruction with Multiple Walsh matrices}
  \date{}
  \author{Enrico Au-Yeung}
\address{ Enrico Au-Yeung\\
          Department of Mathematical Sciences \\
         DePaul University\\
         Chicago}
 \email{eauyeun1@depaul.edu}
 
\begin{document}

  \begin{abstract}
  The problem of how to find a sparse representation of a signal is an important one in applied and computational harmonic analysis.
It is closely related to the problem of how to reconstruct a sparse vector from its projection in a much lower-dimensional vector space.
This is the setting of compressed sensing, where the projection is given by a matrix with many more columns than rows.   
We introduce a class of random matrices that can be used to reconstruct
sparse vectors in this paradigm.  These matrices satisfy the restricted isometry property with overwhelming probability.  We also discuss an application in dimensionality reduction where we initially discovered this class of matrices.
\end{abstract}


\maketitle

 \section{Introduction and Motivation}
 
 In an influential survey paper by Bruckstein, Donoho, and Elad, the problem of finding a sparse solution to an underdetermined linear system is discussed in great detail \cite{BrucksteinDonohoElad2009}.  This is an important problem in applied and computational harmonic analysis.
 Their survey provides plenty of inspiration for future directions of research, with both theoretical and practical consideration.  To make this presentation complete, we provide a brief overview.\\
 
 To motivate our discussion, we start by reviewing how sparsity and redundancy are brought to use.  
 Suppose we have a signal which we regard as a nonzero vector $y \in \mathbb{R}^n$ and there are two available 
orthonormal bases $\Psi$ and $\Phi$.  Then the vector can be expressed as a linear combination of the columns of $\Psi$ or as a linear
combination of the columns of $\Phi$,  \[ y =  \Psi \alpha = \Phi \beta. \]
An important example is to take $\Psi$ to be the identity matrix, and $\Phi$ to be the matrix for discrete cosine transform.
In this case,   $\alpha$ is the representation of the signal in the time domain (or space domain) and $\beta$ is the representation in the frequency domain.
For some pairs of orthonormal bases, such as the ones we have just mentioned, either the coefficients $\alpha$ can be sparse, or the $\beta$ can be sparse, but they cannot both be sparse.  This interesting phenomenon is sometimes called the Uncertainty Principle $\colon$
\[ \left\|   \alpha \right\|_{0} + \left\|   \beta  \right\|_{0} \geq 2 \sqrt{n}. \]
Here, we have written $ \left\|   \alpha \right\|_{0}$ to denote the sparsity of $\alpha$, which is the number of nonzero entries in the vector.  This means that
a signal cannot have fewer than $\sqrt{n}$ nonzero entries in both the time domain and the frequency domain.
Since the signal is sparse in either the time domain or the frequency domain, but not in both, this leads to the idea of combining the two bases by concatenating the two matrices into one matrix $A = \left[ \Psi \ \Phi \right]$. \\

   By a representation of the signal $y$, we mean a column vector $x$ so that $y = Ax$.  The representation of the signal is not unique because the column vectors of $A$ are not linearly independent.  From this observation, we are naturally led to consider a matrix $A$ formed by combining more than two bases.  The hope is that among the many possible ways of representing the signal $y$, there is at least one representation that is very sparse, i.e. most entries of $x$ are zero.  We want the vector $x$ to be $s$-sparse,  which means that at most $s$ of the entries are nonzero.  A natural question that arises is how to find the sparse representation of a given signal $y$.\\
   
    There is a closely related problem  that occurs commonly in signal and image processing.   Suppose we begin with a vector $x \in \mathbb{R}^N$ that is $s$-sparse, which we consider to be our compressible signal.  Using the matrix $A \in \mathbb{R}^{n \times N}$,  we observe the vector $y$ from the projection $y = Ax$. This leads to the following problem: given a matrix $A \in \mathbb{R}^{n \times N}$, where typically $N$ is much larger than $n$, and given $y \in \mathbb{R}^n$,
how to recover the $s$-sparse vector $x  \in \mathbb{R}^N$ from the observation
$y = A x$.   The term most commonly used in this setting is compressed sensing.   This problem is NP-hard, i.e. the natural approach to consider all possible $s$-sparse vectors in $\mathbb{R}^N$ is not feasible. 
The reconstruction of the vector $x$ is accomplished by a non-linear operator $\Delta \colon \mathbb{R}^N \rightarrow \mathbb{R}^n$ that solves the minimization problem,
\[(P1) \quad \min \| x \|_1 \quad \text{ subject to } y = Ax. \]
The following definition plays a central role in this paper.
\begin{defn}
A matrix $A \in \mathbb{R}^{m \times N}$ is said to have the restricted isometry property (RIP) of order s and level $\delta_{s} \in (0,1)$ if
\[  (1- \delta_{s}) \| x \|_{2}^{2} \leq \| A x \|_{2}^{2} \leq (1 + \delta_{s}) \| x \|_{2}^2 \quad \mbox{ for all s-sparse } x \in \mathbb{R}^{N}.  \]
\end{defn}
The restricted isometry property says that the columns of any sub-matrix with at most $s$ columns are close to being orthogonal to each other.
If the matrix $A$ satisfies this property, then the solution to (P1) is unique, i.e. it is possible to reconstruct the $s$-sparse vector  by minimizing the $l_1$ norm of $x$, subject to $y = Ax$.
For this reason, matrices that satisfy the RIP play a key role in compressed sensing.  Some examples of random matrices that satisfy the RIP are the Gaussian, Bernoulli, or partial random Fourier
matrices.  

 From the foundational papers of Donoho \cite{Donoho2006} and Candes, Romberg, and Tao \cite{CandesRombergTao2006a, CandesRombergTao2006b}, the field of compressed sensing has been studied and extended by many others to include a broad range of theoretical issues and applications; see, for example, \cite{Candes2008, CohenDahmenDeVore2009, BaraniukDavenportDeVoreWakin2008, BlumensathDavies2009, MendelsonPajorJaegermann2008, BaraniukCevherDuarte2010, DonohoTanner2009, Foucart2010, Tropp2004, Tropp2006, Rauhut2007, RudelsonVershynin2008, RauhutWard2012, KrahmerWard2012, TroppRombergBaraniuk2010}, and the comprehensive treatment found in \cite{FoucartRauhut2013}.\\

  The search for structured matrices that can be used in compressed sensing continues to be an active research area (see, e.g., \cite{Rauhut2010}.) Towards that goal, our contribution is to introduce a class of random matrices that satisfy the RIP with overwhelming probability.  We also describe an application where we initially discovered this class of matrices. \\

\subsection{Application}
Dimensionality reduction is another area where matrices that satisfy the RIP play an important role.  
A powerful tool is the Johnson-Lindenstrauss (JL) lemma.  This lemma tells us that the distance between each pair of points in a high-dimensional space is nearly preserved if we project the points into a much lower-dimensional space using a random linear mapping.   Krahmer and Ward \cite{KrahmerWard2012} showed that if a matrix satisfies the RIP, then we can use it to create such a mapping if we randomize the column signs of the matrix.  For a precise statement, see \cite{KrahmerWard2012}.  Together with our matrix that satisfies the RIP,  their result allows one to create a matrix to be used in a JL-type embedding.  To demonstrate this in a concrete setting, let us turn to an application in robust facial recognition. \\

The goal of object recognition is to use training samples from $k$ distinct object classes to determine the class to which a new sample belongs.  We arrange the given $n_j$ training samples from the $j$-th class as columns of a matrix $Y_j \equiv [v_{j,1}, v_{j,2}, \ldots, v_{j, n_j} ] \in \mathbb{R}^{m \times n_j}$.   In the context of a face recognition system, we identify a $w \times h$ facial image with the vector $v \in \mathbb{R}^{m}$ ($m = wh$) given by stacking its columns.  Therefore, the columns of $Y_j$ are the training facial images of the $j$-th person.  One effective approach for exploiting the structure of the $Y_j$ in object recognition  is to model the samples from a single class as lying on a linear subspace.  Subspace models are flexible enough to capture the structure in real data, where it has been demonstrated that the images of faces under varying lighting and expressions lie on a low-dimensional subspace \cite{BasriJacobs1997}.   For our present discussion, we will assume that the training samples from a single class do lie on a single subspace.\\

Suppose we are given sufficient training samples of the $j$-th object class, \[Y_j \equiv [v_{j,1}, v_{j,2}, \ldots, v_{j, n_j} ] \in \mathbb{R}^{m \times n_j}.\]  Then,  any new  sample $y_{new} \in \mathbb{R}^{m}$ from the same class will approximately lie in the linear span of the training samples associated with object $j$, i.e.
\[ y_{new} = c_{j,1} v_{j,1} + c_{j,2} v_{j,2} + \ldots c_{j, n_j} v_{j, n_j}, \]
for some coefficients $c_{j,k} \in \mathbb{R}, 1 \leq k \leq n_{j}$.  We define a new matrix $\Phi$ for the entire training set as the concatenation of the $n$ training samples of all $k$ object classes, \[\Phi = [Y_1, Y_2, Y_3, \ldots, Y_k] = [v_{1,1}, v_{1,2}, \ldots, v_{2,1}, v_{2,2}, \ldots, v_{k, n_k}]. \]
The new sample $y_{new} \in \mathbb{R}^m$ can be expressed as a linear combination of all training samples, $y_{new} = \Phi x$, where   the transpose of the vector $x$ is of the form,
\[ x = [0, 0, \ldots, 0, c_{j,1}, c_{j,2}, \ldots, c_{j, n_j}, 0, 0, \ldots, 0] \in \mathbb{R}^n,\]
i.e. $x$ is the coefficient vector whose entries are zero, except those entries associated with the $j$-th class.   The sparse vector $x$ encodes the identity of the new sample $y_{new}$.  The task of classifying a new sample amounts to solving the linear system $y_{new} = \Phi x$ to recover the sparse vector $x$.  
For more details, see \cite{WrightYangSastryMa2009}, where the authors  presented strong experimental evidence to support this approach to robust facial recognition.\\

One practical issue that arises is that for face images without any pre-processing, the corresponding linear system $y = \Phi x$ is very large.  For example, 
if each face image is given at a typical resolution of $640 \times 480$ pixels, then the matrix $\Phi$ has $m$ rows, where $m$ is in the order of $10^{5}$.  Using scalable algorithms, such as linear programming, applying this directly to high-resolution images still requires enormous computing power.  Dimensionality reduction becomes indispensable in this setting. The projection from the image space to the much lower-dimensional  feature space can be represented by a matrix $P$,  where $P$ has many more columns than rows.  The linear system $y = \Phi x$ then becomes
\[ \widetilde{y} \equiv Py = P \Phi x.\]
The new sample $y$ is replaced by its projection $\widetilde{y}$. The sparse vector $x$ is reconstructed by solving the minimization problem,
\[  \min \| x \|_1 \quad \text{ subject to } y = P \Phi x. \]
In the past, enormous amount of effort was spent to develop feature-extraction methods for finding projections of images into lower-dimensional spaces.  Examples of feature-extraction methods include EigenFace, FisherFace, and a host of creative techniques; see, e.g. \cite{HespanhaKriegman1997}.
  For the approach to facial recognition that we have described, choosing a matrix $P$ is no longer a difficult task.  We can select a matrix $P$ so that it nearly preserves the distance between every pair of vectors, i.e. $\|Px - P y\|_2 \approx \| x - y \|_2$.  As mentioned earlier,  beginning with a matrix $A$ that satisfies the RIP,  the result of Krahmer and Ward allows one to create a matrix $P$ to be used in a JL-type embedding. 



\subsection{Notation}\label{Notation}
Before continuing further, we need to define some  terminology.  The Rademacher system $\{r_{n}(x)\}$ on the interval $[0,1]$ is a set of orthogonal functions defined by
\[ r_{n}(x) = \mbox{sign}( \sin(2^{n+1} \pi x)) ; \quad n = 0, 1, 2, 3, \ldots\]
The Rademacher system does not form a basis for $L^2([0,1])$, but this can be remedied by considering the Walsh system of functions.  
Each Walsh function is a product of Rademacher functions.  
The sequence of Walsh functions is defined as follows.  Every positive integer $n$ can be written in the binary system as:
\[ n = 2^{n_1} + 2^{n_2} +  \ldots + 2^{n_k},\]
where the integers $n_j$ are uniquely determined by $n_{j+1} < n_{j}.$
The Walsh functions $\{W_n(x)\}_{n=0}^{\infty}$ are then given by
\[W_{0}(x) = 1, \ W_{n}(x) = r_{n_1}(x) r_{n_2}(x) \ldots r_{n_k}(x). \]
The Walsh system forms an orthogonal basis for $L^2([0,1])$.
There is a convenient way to represent these functions as vectors.  Define the matrices
$H_{0} = 1,$ and for $n\geq 1$,   \[ H_{n} = \frac{1}{\sqrt{2}}
    \left[ \begin{array}{cc}
             H_{n-1} & -H_{n-1}   \\
             H_{n-1} & H_{n-1}
             \end{array}  \right].
                   \]
Then, the column vectors of $H_n$ form an orthogonal basis on $R^{2^n}$.  Note that the matrix $H_n$ has $2^n$ rows and $2^n$ columns.    Because of its close connection to the Walsh system, a matrix of the form $H_n$ is called a Hadamard-Walsh matrix.\\
 
The inner product of two vectors $x$ and $y$ is denoted by $\langle x, y \rangle$. The Euclidean norm of a vector $x$ is denoted by $\| x \|_{2}.$
  If a vector has at most s nonzero entries, 
we say that the vector is s-sparse.  For clarity, we often label constants by $C_1, C_2, \ldots$, but we do not keep track of their precise values. \\

For a matrix $A \in \mathbb{R}^{m \times N}$, its operator norm is
 $ \| A \| = \sup \{ \| Ax \|_{2} \colon \|x\|_{2} = 1 \} $.   If $x \in \mathbb{R}^N$, then  we say that $\Gamma$ is the support set of the vector if the  entries of $x$ are nonzero only on the set $\Gamma$, and we write $\mbox{supp}(x) = \Gamma$.  
  We define
$B_{\Gamma} = \{ x \in \mathbb{R}^{N} \colon \|x\|_{2} = 1, \ \mbox{supp}(x) = \Gamma\}.$ We write $A^{\ast}$ for the adjoint (or transpose) of the matrix.   Working with s-sparse vectors, there is another norm defined by
 \[ \| A \|_{\Gamma} = \sup \{ | \langle Ax, y \rangle | \colon  x \in B_{\Gamma},  \ y \in B_{\Gamma}, \ | \Gamma | \leq s \}. \]
 This norm is important because if the matrix $A$ obeys the relation
$ \| I - A^{\ast}A \|_{\Gamma} \leq \delta_{s}, $
then  $A$ satisfies  the RIP of order s and level $\delta_s$.\\

Let us introduce a model called Sparse City. 
From now on, we fix a positive integer $m$ that is a power of 2, i.e. $m = 2^{k}$, and focus on a single Hadamard-Walsh matrix $W$ with $m$ rows and $m$ columns.  Let $W_{n}^m$ be the $m \times n$ matrix formed by selecting the first $n$ columns of the matrix $W$. 
 Let $\Theta$ be a bounded random variable with expected value zero, i.e.
\[ E(\Theta) = 0, | \Theta | \leq B. \]
To be precise, the random variable $\Theta$ is equally likely to take one of the four possible values, $\{1, -1, 3, -3\}$.
Define the random vectors $x_1, x_2, \ldots x_b \in \mathbb{R}^{m}$, so that the entries of each vector are independent random variables drawn from the same
probability distribution as $\Theta$.  
For each vector $x_j = (\theta_{j1}, \theta_{j2}, \ldots, \theta_{jm})$, we have $E(\theta_{jw}) = 0$ and $| \theta_{jw} | \leq B,$ for $1 \leq w \leq m$.
We associate a matrix $D_j$ to each vector $x_j$, so that each $D_j \in \mathbb{R}^{m \times m}$ 
is a diagonal matrix with the entries of $x_j$ along the diagonal.
To construct the matrix $A$, we concatenate $b$ blocks of $D_{j} W_{n}^{m}$, so that written out in block form,
 \[
A = \left[ \begin{array}{ccccc}
               D_1 W_{n}^{m} \ | & D_2 W_{n}^{m} \ | & D_3 W_{n}^{m} \ | & D_4 W_{n}^{m} \ | \ \ldots \ldots  \ | & D_{b} W_{n}^{m} \end{array} \right].
   \]
 Note that the matrix $A$ has $m$ rows and $nb$ columns. In our application, Walsh matrices are more appropriate than other orthogonal matrices, such as discrete cosine transforms (DCT).  For illustration, if $A \in \mathbb{R}^{1024 \times 20480}$, with $b = 320$ blocks, then each entry of $64 A$ is one of the four values $\{1, -1, 3, -3\}$.  Consider any vector $y$ that contains only integer values, ranging from 0 to 255, which are typical for facial images.  The product $Ay$ can be computed from $64 \times Ay$; the calculation of $64 \times Ay$ uses only integer-arithmetic operations.


\subsection{Main results}
Our first result is that the matrix  satisfies the RIP in expectation.

\begin{thm}\label{MainTheorem1}
Let $W$ be the Hadamard-Walsh matrix with $m$ rows and $m$ columns.  Let $W_{n}^m$ be the $m \times n$ matrix formed by selecting the first $n$ columns of the matrix $W$. 
The matrix $A \in \mathbb{R}^{m \times nb}$ is constructed by concatenating $b$ blocks of $D_{j} W_{n}^{m}$, so that
 \[
A = \left[ \begin{array}{ccccc}
               D_1 W_{n}^{m} \ | & D_2 W_{n}^{m} \ | & D_3 W_{n}^{m} \ | & D_4 W_{n}^{m} \ | \ \ldots \ldots  \ | & D_{b} W_{n}^{m} \end{array} \right].
   \]
   Each $D_j \in \mathbb{R}^{m \times m}$ is a diagonal matrix, as defined in section (\ref{Notation}).
Then, there exists a constant $C > 0$ such that for any $0 < \delta_{s} \leq 1$,  we have
\begin{equation*}
E\| I - A^{\ast}A \|_{\Gamma} \leq \delta_{s}
\end{equation*}
provided that \[m \geq C \cdot \delta_{s}^{-2} \cdot s \cdot \log^{4}(nb) \] and $m \leq nb$. \\ More precisely, there are constants $C_{1}$ and $C_{2}$ so that 
\begin{equation}\label{eqn:MainThmNormBound}
 E\| I - A^{\ast}A \|_{\Gamma} \leq \sqrt{\frac{C_{1} \cdot s \cdot \log^{2}(s)  \cdot \log(mb) \cdot \log(nb)}{m}}
\end{equation} provided that
\[ m \geq C_{2} \cdot s \cdot \log^{2}(s) \log(mb) \log(nb). \]

\end{thm}

 The next theorem tells us that  the matrix 
satisfies the restricted isometry property with overwhelming probability.

\begin{thm}\label{Thm:TailBound}
     Fix a constant $\delta_s > 0$.  Let $A$ be the matrix specified in Theorem \ref{MainTheorem1}.
 Then, there exists a constant $C > 0$ such that  
  \[ \| I - A^{\ast}A \|_{\Gamma} < \delta_s   \]
  with probability at least $1 - \epsilon$ provided that
  \[ m \geq C \cdot \delta_{s}^{-2} \cdot s \cdot \log^{4}(nb) \cdot \log(1/ \epsilon) \] and $m \leq nb$.  
\end{thm}

\subsection{Related Work}
Gaussian and Bernoulli matrices satisfy the restricted isometry property (RIP) with overwhelmingly high probability, provided that the number of measurements $m$  satisfies $m = O( s \log(\frac{N}{s}))$. 
Although these matrices require the least number of measurements, they have limited use in practical applications.  Storing an unstructured matrix, in which all the entries of the matrix are independent of each other, requires a prohibited amount of storage.  From a computation and application view point, this has motivated the need to find structured random matrices that satisfy the RIP.  Let us review three of the most popular classes of random matrices that are appealing alternatives to the Gaussian matrices. For a broad discussion and other types of matrices, see \cite{LiCong2015} and \cite{Rauhut2010}. \\

The random subsampled Fourier matrix is constructed by randomly choosing $m$ rows from the $N \times N$ discrete Fourier transform (DFT) matrix.  In this case, it is important to note that the fast Fourier transform (FFT) algorithm can be used to significantly speed up the matrix-by-vector multiplication. A random subsampled Fourier matrix with $m$ rows and $N$ columns satisfies the RIP with high probability, 
provided that $m \geq C \cdot \delta^{2} s \log(N)^4$, where $C$ is a universal constant; see \cite{RudelsonVershynin2008} for a precise statement. \\

The next type of structured random matrices are partial random Toeplitz and circulant matrices.  These matrices naturally arise in applications where convolutions are involved. Recall that for a Toeplitz matrix, each entry $a_{ij}$ in row $i$ and column $j$  is determined by the value of $i - j$, so that for example, $a_{11} = a_{22} = a_{33}$ and $a_{21} = a_{32} = a_{43}$. To construct a random $m \times N$ Toeplitz matrix $A$, only $N + m  -2$ random numbers are needed.  Haupt, Bajwa, Raz, Wright and Nowak \cite{BajwaNowak2007} showed that the matrix $A$ satisfies the RIP of order $3s$ with high probability for every $\delta \in (0, \frac{1}{3})$, provided that $m > C \cdot s^3 \log(\frac{N}{s})$, where $C$ is a constant. \\

There are many situations in signal processing where we encounter signals that are band-limited and are sparse in the frequency domain. 
 The random demodulator matrix is suitable in this setting \cite{TroppRombergBaraniuk2010}.
 For motivation, imagine that we try to acquire a single high-frequency tone that lies within a wide spectral band.  Then, a low-rate sampler with an antialiasing filter will be oblivious to any tone whose frequency exceeds the passband of the filter.  To deal with this problem, the random demodulator smears the tone across the entire spectrum so that it leaves a signature that a low-rate sampler can detect.   Consider a signal whose highest frequency does not exceed $\frac{W}{2}$ hertz.  We can give a mathematical description of the system. Let $D$ be a $W \times W$ diagonal matrix, with random numbers along the diagonal.  Next, we consider the action of the sampler and suppose the sampling rate is $R$, where $R$ divides $W$.  Each sample is then the sum of $\frac{W}{R}$ consecutive entries of the demodulated signal.  The action of the sampling is specified by a matrix $G$ with $R$ rows and $W$ columns, such that the $r$-th row has $\frac{W}{R}$ consecutive ones, beginning in column $(\frac{rW}{R}) + 1$ for each $r = 0, 1, 2, \ldots, R-1$.   For example, when $W = 12$ and $R = 3$, we have
        \[ G = 
    \left[ \begin{array}{cccccccccccc}
             1 & 1 & 1 & 1 & 0 & 0 & 0 & 0 & 0 & 0 & 0 & 0 \\
             0 & 0  & 0 & 0 & 1 & 1 & 1 & 1 & 0 & 0 & 0 & 0\\
             0 & 0  & 0 & 0 & 0 & 0 & 0 & 0 & 1 & 1 & 1 & 1 
             \end{array}  \right].
                   \]
Define  the $R \times W$ matrix $A = GDF$, where $F$ is the $W \times W$ discrete Fourier transform (DFT) matrix with the columns permuted; see \cite{TroppRombergBaraniuk2010} for further detail. For a fixed $\delta > 0$, if the sampling rate $R$ is greater than or equal to $C \delta^{-2} \cdot s \log(W)^6$, then an $R \times W$ random demodulation matrix $A$ has the RIP of order $s$ with constant $\delta_s \leq \delta$, with probability at least $1 - O(\frac{1}{W})$. \\

In contrast to the classes of matrices described above, the class of structured random matrices introduced in Sparse City has a block form.
The matrix $A \in \mathbb{R}^{m \times nb}$ is constructed by concatenating $b$ blocks.  More precisely,
 \[
A = \left[ \begin{array}{ccccc}
               D_1 W_{n}^{m} \ | & D_2 W_{n}^{m} \ | & D_3 W_{n}^{m} \ | & D_4 W_{n}^{m} \ | \ \ldots \ldots  \ | & D_{b} W_{n}^{m} \end{array} \right].
   \]
  In each block of $D_{j} W_{n}^{m}$, the same $m \times n$ matrix $W_n^{m}$ is used, but each block has its own random diagonal matrix $D_j$.
  For compressed sensing to be useful in applications, we need to have suitable hardware and a data acquisition system.  In seismic imaging, the signals  are often measured by multiple sensors.  
  A signal can be viewed as partitioned into many parts.   Different sensors  are responsible for measuring different parts of the signal. Each sensor is equipped with its own scrambler which it uses to randomly scramble the measurements. The block structure of the sensing matrix $A$ facilitates the design  of a suitable data acquisition scheme tailored to this setting.

\newpage
\section{Mathematical tools}
We collect together the tools we need to prove the main results.  We begin with a fundamental result by Rudelson and Vershynin \cite{RudelsonVershynin2008}, followed by an extension of this result, and then a concentration inequality.  In what follows, 
for vectors $x, y \in \mathbb{R}^n$, the tensor $x \otimes y$ is the rank-one operator defined by $(x \otimes y)(z) = \langle x, z \rangle y.$  
For a given subset $\Gamma \subseteq \{1, 2, \ldots, n\}$, the notation $x^{\Gamma}$ is the restriction of the vector $x$ on the coordinates in the set $\Gamma$.
\begin{lem}\label{lem:RudelsonVershynin}(Rudelson and Vershynin)

Let $x_1, x_2, x_3, \ldots, x_{m}$, with $m \leq n$, be vectors in $\mathbb{R}^n$ with uniformly bounded entries, $\| x_{i} \|_{\infty} \leq K$ for all $i$.  Then
\begin{equation}\label{eqn:RVlemma}
E \sup_{| \Gamma | \leq s} \left\|    \sum_{i=1}^m \epsilon_{i} \ x_{i}^{\Gamma} \otimes   x_{i}^{\Gamma}   \right\|  \leq M \cdot  \sup_{| \Gamma | \leq s}  \left\|    \sum_{i=1}^m  x_{i}^{\Gamma} \otimes   x_{i}^{\Gamma}   \right\|^{1/2}
\end{equation}
where the constant $M$  equals $C_1(K) \sqrt{s} \log(s) \sqrt{\log n} \sqrt{\log m}$.
\end{lem}
Since our next lemma is an extension of this lemma, we provide a review of the main ideas in the proof of Lemma (\ref{lem:RudelsonVershynin}).\\

Let $E_1$ denote the left-hand side of (\ref{eqn:RVlemma}).  
We will bound $E_1$ by the supremum of a Gaussian process.  
Let $g_1, g_2, g_3, \ldots, g_m$ be independent standard normal random variables. The expected value of $\left| g_i \right|$ is a constant that does not depend on the index $i$.
\begin{align*}
E_1 & \leq C_3 \cdot  E \sup_{| \Gamma | \leq s} \left\|    \sum_{i=1}^m E \left| g_{i} \right| \ \epsilon_i \ x_{i}^{\Gamma} \otimes   x_{i}^{\Gamma}   \right\|  \\
& \leq C_3 \cdot  E \sup_{| \Gamma | \leq s} \left\|    \sum_{i=1}^m  \left| g_{i} \right|  \ x_{i}^{\Gamma} \otimes   x_{i}^{\Gamma}   \right\| \\
&  = C_3 \cdot E \sup \left\{\left|   \sum_{i=1}^{m} g_{i} \langle   x_i, x \rangle^{2}     \right| \colon | \Gamma | \leq s, x \in B_{\Gamma}  \right\} 
\end{align*}

To see that the last equality is true, consider an operator $A$ on $\mathbb{R}^n$ defined by \[ A z  = \sum_{i=1}^m g_i \langle x_i , z \rangle x_i \]  and since $A$ is a self-adjoint operator,  it follows that \[ \| A \|_{op} = \sup_{\|z\| = 1} \langle Az, z\rangle =\sup_{ \| z \| = 1} \sum_{i=1}^m g_i \langle x_i,z \rangle^2. \]

For each vector $u$ in $\mathbb{R}^n$, we consider the Gaussian process
\[G(u) = \sum_{i=1}^m g_i \langle x_i, u \rangle^2.\] 
This Gaussian process is a random process indexed by vectors in $\mathbb{R}^n$. \\

Thus to obtain an upper bound on $E_1$, we need an estimate on the expected value of the supremum of a Gaussian process over an arbitrary index set.  We use Dudley's Theorem (see \cite{Talagrand1996}, Proposition 2.1) to obtain an upper bound. 
\begin{thm}
Let $\left(X(t): t \in T \right)$ be a Gaussian process with the associated pseudo-metric $d(s,t) = \left(  E \left | X(s) - X(t) \right |^2 \right)^{1/2}$.  Then there exists a constant $K > 0$ such that
\[ E \sup_{t \in T} X(t) \leq K \int_{0}^{\infty} \sqrt{\log N (T, d, u)  } \ du.\]
Here, $T$ is an arbitrary index set, and the covering number $N(T, d, u)$ is the smallest number of balls of radius $u$ to cover the set $T$ with respect to the pseudo-metric.

\end{thm}
By applying Dudley's inequality with
\[ T = \bigcup_{\left| \Gamma \right| \leq s} B_{\Gamma} \]
  the above calculations show that, 
\begin{equation}\label{eqn:coveringNumber}
E_1 \leq C_4 \int_{0}^{\infty} \left[  \log N \left( \bigcup_{ | \Gamma | \leq s} B_{\Gamma}, \| \cdot \|_{G}, u \right) \right]^{1/2} \ du,
\end{equation}
where $N$ is the covering number. \\

There is a semi-norm associated with the Gaussian process, so that if $x$ and $y$ are any two fixed vectors in $\mathbb{R}^n$, then
\begin{eqnarray*}
\| x - y\|_G & = & \left( E \left|  G(x) - G(y) \right|^2  \right)^{1/2}   \\
     & = & \left[ \sum_{i=1}^m \left( \langle x_i, x \rangle^2   - \langle x_i, y  \rangle^2  \right )^2 \right]^{1/2} \\
     &\leq & \left[ \sum_{i=1}^m \left( \langle x_i, x \rangle + \langle x_i, y \rangle \right)^2 \right]^{1/2} \cdot \max_{i \leq m} \left |   \langle x_i, x - y  \rangle \right | \\
     & \leq & 2 \max_{ | \Gamma | \leq s, z \in \mathbb{B}_{2}^{ \Gamma } } \left[   \sum_{i=1}^m \langle    x_i, z \rangle^2    \right]^{1/2} \cdot \max_{i \leq m} \left | \langle  x_i, x - y \rangle  \right | =  2 \ R \ \max_{i \leq m} \left | \langle  x_i, x- y \rangle \right |,
\end{eqnarray*}
where
\[R \equiv \sup_{| \Gamma | \leq s}  \left\|    \sum_{i=1}^m  x_{i}^{\Gamma} \otimes   x_{i}^{\Gamma}   \right\|^{1/2}  .\]



Thus, by a change of variable in the integral in (\ref{eqn:coveringNumber}), we see that
\begin{equation}\label{eqn:Dudley_inequality2}
E_3 \leq C_5 R \sqrt{s} \int_{0}^{\infty} \log^{1/2} N \left( \frac{1}{\sqrt{s}} \bigcup_{|T| \leq s} B_{\Gamma}, \| \cdot \|_X, u \right) \ du.
\end{equation}
Here, the semi-norm $\| x \|_X$ is defined by
 \[ \| x \|_X = \max_{i \leq m} \left | \langle x_i, x \rangle \right |. \]
 
It is sufficient to show  the integral in (\ref{eqn:Dudley_inequality2}) is bounded by $C_{11}(K)  \cdot \log(s) \cdot \sqrt{\log n} \cdot \sqrt{\log m}.$ \\

This concludes our review of the main ideas in the proof of Lemma (\ref{lem:RudelsonVershynin}). \\

We extend the fundamental lemma of Rudelson and Vershynin.  The proof follows the strategy of the proof of the original lemma, with an additional ingredient.
The Gaussian process involved is replaced by a tensorized version, with the appropriate tensor norm.
\begin{lem}\label{lemma:Extension}(Extension of the fundamental lemma of Rudelson and Vershynin)

Let $u_1, u_2, u_3, \ldots, u_{k}$, and $v_1, v_2, v_3, \ldots, v_{k}$, with $k \leq n$, be vectors in $\mathbb{R}^n$ with uniformly bounded entries, $\| u_{i} \|_{\infty} \leq K$ and $\| v_{i} \|_{\infty} \leq K$ for all $i$.  Then  

\begin{equation}\label{eqn:ExtendedRVlemma}
E \sup_{| \Gamma | \leq s} \left\|    \sum_{i=1}^k \epsilon_{i} \ u_{i}^{\Gamma} \otimes   v_{i}^{\Gamma}   \right\|  \leq M \cdot \left( \sup_{| \Gamma | \leq s}  \left\|    \sum_{i=1}^k  u_{i}^{\Gamma} \otimes   u_{i}^{\Gamma}   \right\|^{1/2} + \sup_{| \Gamma | \leq s}  \left\|    \sum_{i=1}^k  v_{i}^{\Gamma} \otimes   v_{i}^{\Gamma}   \right\|^{1/2} \right)
\end{equation}
where the constant $M$ depends on $K$ and the sparsity $s$.
\end{lem}

\begin{proof}
Let $E_1$ denote the left-hand side of (\ref{eqn:ExtendedRVlemma}).   Our plan is to bound $E_1$ by the supremum of a Gaussian process.
Let $g_1, g_2, g_3, \ldots, g_k$ be independent standard normal random variables.  Then
\begin{align*}
E_1 & = E \sup \left\{\left|   \sum_{i=1}^{k} \epsilon_{i} \langle   x_p, u_i \rangle  \langle  v_i, x_q  \rangle      \right| \colon | \Gamma | \leq s, x_p \in B_{\Gamma}, x_q \in B_{\Gamma}  \right\} \\
& \leq C_3 \cdot \mathbb{E} \sup \left\{\left|   \sum_{i=1}^{k} g_{i} \langle   x_p, u_i \rangle  \langle  v_i, x_q  \rangle      \right| \colon | \Gamma | \leq s, x_p \in B_{\Gamma}, x_q \in B_{\Gamma}  \right\} 
\end{align*}
When $G(x)$ is a Gaussian process indexed by the elements $x$ in an arbitrary index set $T$, Dudley's inequality states that
\[ E \sup_{x \in T} \left|  G(x)    \right| \leq C \cdot \int_{0}^{\infty} \log^{1/2} N(T, d, u) \ du, \]
 with the pseudo-metric $d$ given by
\[ d(x,y) = \left(  E \left|  G(x) - G(y)       \right|^{2}  \right)^{1/2}.   \]
Our Gaussian process is indexed by two vectors $x_p$ and $x_q$ so that 
\[ G(x_p, x_q) = \sum_{i} g_{i} \langle   x_p, u_i \rangle  \langle  v_i, x_q  \rangle   \]
and the index set is \[T = \bigcup_{|   \Gamma |  \leq s} B_{\Gamma} \otimes    B_{\Gamma}.   \] The pseudo-metric on $T$ is given by
\begin{align*}
& d\left(   ( x_p, x_q), (y_p, y_q)    \right) \\
& = \left[   \sum_{i=1}^{k} \left(  \langle x_p, u_i \rangle \langle v_i, x_q \rangle - \langle y_p, u_i \rangle \langle v_i, y_q \rangle        \right)^2       \right]^{1/2} \\
& = \frac{1}{2} \left[  \sum_{i=1}^k \left(   \langle x_p + y_p, u_i \rangle \langle v_i, x_q - y_q \rangle + \langle x_p - y_p, u_i \rangle \langle v_i, x_q + y_q \rangle     \right)       \right]^{1/2} \\
& \leq \frac{1}{2} \cdot \max_{i} \left(   \left| \langle u_i, x_p - y_p \rangle   \right|,     \left| \langle v_i, x_q - y_q   \rangle \right|      \right) \cdot \left[  \sum_{i=1}^{k} \left(    \left| \langle x_p + y_p , u_i \rangle   \right| + \left| \langle x_q + y_q , v_i \rangle   \right|     \right)^{2}    \right]^{1/2} \\
& \leq Q \cdot \max_{i} \left(    \left| \langle u_i, x_p - y_p \rangle   \right|,     \left| \langle v_i, x_q - y_q   \rangle \right|      \right),
\end{align*}
where  the quantity $Q$ is defined by
\[
Q  = \frac{1}{2} \sup \left\{ \left[ \sum_{i=1}^{k} \left(  \left|  \langle x_p + y_p, u_i \rangle  \right| +  \left|  \langle x_q + y_q, v_i \rangle  \right|  \right)^{2} \right]^{1/2} \colon (x_p, x_q) \in \Gamma    \right\}
\]
We bound the quantity $Q$ in the following calculations.
\begin{align*}
& Q^2  = \frac{1}{4} \sup_{ (x_p, x_q) \in \Gamma}  \sum_{i=1}^{k} \left(  \left|  \langle x_p + y_p, u_i \rangle  \right| +  \left|  \langle x_q + y_q, v_i \rangle  \right|      \right)^{2}   \\
\leq &\frac{1}{4} \sup_{(x_p, x_q) \in \Gamma} \left\{    \sum_{i=1}^{k} \left| \langle x_p + y_p, u_i  \rangle \right|^2 +   \sum_{i=1}^{k} \left| \langle x_p + y_p, v_i   \rangle \right|^2  + 2 \sum_{i=1}^{k}  \left|    \langle x_p + y_p, u_i  \rangle    \right| \cdot  \left|    \langle x_q + y_q, v_i  \rangle    \right|  \right\} \\
 \leq & \left\| \sum_{i =1}^{k} u_{i} \otimes u_{i} \right\|_{\Gamma}   +    \left\| \sum_{i =1}^{k} v_{i} \otimes v_{i} \right\|_{\Gamma}        + \\
        & \hspace{1in} \frac{1}{2} \left[\sup_{(x_p, x_q) \in \Gamma} \left(   \sum_{i=1}^{k} \left| \langle x_p + y_p, u_i  \rangle \right|^2       \right)^{1/2} +  \left(   \sum_{i=1}^{k} \left| \langle x_q + y_q, v_i  \rangle \right|^2       \right)^{1/2} \right] \\
 \leq &  \left\| \sum_{i =1}^{k} u_{i} \otimes u_{i} \right\|_{\Gamma}   +    \left\| \sum_{i =1}^{k} v_{i} \otimes v_{i} \right\|_{\Gamma}   + 2  \left\| \sum_{i =1}^{k} u_{i} \otimes u_{i} \right\|_{\Gamma}^{1/2} \cdot  \left\| \sum_{i =1}^{k} u_{i} \otimes u_{i} \right\|_{\Gamma}^{1/2} \\
 = & \left[   \left\| \sum_{i =1}^{k} u_{i} \otimes u_{i} \right\|_{\Gamma}^{1/2} + \left\| \sum_{i =1}^{k} v_{i} \otimes v_{i} \right\|_{\Gamma}^{1/2}    \right]^{{\Large{2}}} \equiv S^2. \\
\end{align*}
We now define two norms.  Let $\| x \|_{\infty}^{(U)} = \max_{i} \left|   \langle x, u_i \rangle  \right|$ and $\| x \|_{\infty}^{(V)} = \max_{i} \left|   \langle x, v_i \rangle  \right|$. The above calculations show that the pseudo-metric satisfies the next inequality,
\[ d(   (x_p, x_q), (y_p, y_q) ) \leq S \cdot \max \left(   \| x_p - y_p \|_{\infty}^{(U)} \ ,   \| x_q - y_q \|_{\infty}^{(V)}   \right).  \]
Let $\widetilde{T} = \bigcup_{| \Gamma | \leq s} B_{\Gamma}.$  Then $T \subseteq \widetilde{T} \otimes \widetilde{T}$.  Moreover, the covering number of the set $T$ and the covering number of the set $\widetilde{T}$ must satisfy the relation
\[ N(T, d, u) \leq N( \widetilde{T} \otimes \widetilde{T}, \widetilde{d}, u). \]
Here, $d$ and $\widetilde{d}$ are the pseudo-metrics  for the corresponding index sets.
Consequently, we have
\begin{align*} 
& \int_{0}^{\infty} \log^{1/2} N(T, d, u) \ du \\
&\leq S \cdot \int_{0}^{\infty} \log^{1/2} N( \widetilde{T}, \| \cdot \|_{\infty}^{(U)}, u) \ du + S \cdot \int_{0}^{\infty} \log^{1/2} N( \widetilde{T}, \| \cdot \|_{\infty}^{(V)}, u) \ du. 
\end{align*}
We have completed all the necessary modification to the proof of the original lemma.  The rest of the proof proceeds in exactly the same manner as the proof of the original lemma, almost verbatim, and we omit the repetition.

\end{proof}


In order to show that, with high probability, a random quantity does not deviate too much from its mean, we invoke a concentration inequality for sums of independent symmetric random variables in a Banach space. (See \cite{TroppRombergBaraniuk2010}, Proposition 19, which follows from \cite{Talagrand1991}, Theorem 6.17).
\begin{prop}(Concentration Inequality)\label{Prop:ConcentrationIneq}

Let $Y_1, Y_2, \ldots, Y_R$ be independent, symmetric random variables in a Banach space X.  Assume that each random variable satisfies the bound $\| Y_j \|_{X} \leq B$ almost surely, for $1 \leq j \leq R$.  Let $Y = \| \sum_{j} Y_j\|_{X}.$ Then there exists a constant $C$ so that for all $u, t \geq 1$,
\[ P\left( Y > C [u E(Y) + t B] \right) \leq e^{-u^2} + e^{-t}.\]

\end{prop}

\vspace{0.3in}

We define a sequence of vectors that depend on the entries in the matrix $W_{n}^{m}$.\\
Let $y_{kw} \in \mathbb{R}^{nb}$, where the entries  indexed by $(k-1)n + 1, (k-1)n + 2, (k-1)n + 3, \ldots, kn$ are from row $w$
of the matrix $W_{n}^{m}$, while all other entries are zero.  The next example illustrates the situation.
\begin{example}\label{ex:2}
Consider the matrix $W_{n}^{m}$ with $m$ rows and $n$ columns.  
          \[ W_{n}^{m} = 
    \left[ \begin{array}{cc}
             a(1,1) & a(1,2)   \\
             a(2,1) & a(2,2) \\
             a(3,1) & a(3,2) \\
             a(4,1) & a(4,2)
             \end{array}  \right] 
                   \]

Here, $m = 4$ and $n = 2$.  We define the vectors $y_{11}, y_{12}, y_{13}, y_{14}$ by
\[
y_{11} = \left[ \begin{array}{c}
               a(1,1)  \\ 
               a(1,2) \\
               0 \\ 0 \end{array} \right]               
 y_{12} = \left[ \begin{array}{c}
               a(2,1)  \\ 
               a(2,2) \\
               0 \\ 0 \end{array} \right] 
 y_{13} = \left[ \begin{array}{c}
               a(3,1)  \\ 
               a(3,2) \\
               0 \\ 0 \end{array} \right]   
 y_{14} = \left[ \begin{array}{c}
               a(4,1)  \\ 
               a(4,2) \\
               0 \\ 0 \end{array} \right]   
\]
and we  define the vectors $y_{21}, y_{22}, y_{23}, y_{24}$ by
\[
y_{21} = \left[ \begin{array}{c}
               0 \\ 0 \\
               a(1,1)  \\ 
               a(1,2) \end{array} \right]               
 y_{22} = \left[ \begin{array}{c}
               0 \\ 0 \\
               a(2,1)  \\ 
               a(2,2) \end{array} \right] 
 y_{23} = \left[ \begin{array}{c}
               0 \\ 0 \\
               a(3,1)  \\ 
               a(3,2)  \end{array} \right]  
  y_{24} = \left[ \begin{array}{c}
               0 \\ 0 \\
               a(4,1)  \\ 
               a(4,2)  \end{array} \right]   
\]
Since the columns of $W_{n}^{m}$ come from an orthogonal matrix,  we have the following relations
\[ \sum_{k=1}^{4} \left( a(k,1) \right)^2 = 1, \ \sum_{k=1}^{4} \left( a(k,2) \right)^2 = 1, \ \sum_{k=1}^4 a(k,1)a(k,2) = 0. \]
The rank-one operator $y_{11} \otimes y_{11}$ is defined by $\left(y_{11} \otimes y_{11}\right)(z) = \langle z, y_{kw} \rangle y_{kw}$, for every $z \in \mathbb{R}^4$.  Explicitly in matrix form, this rank-one operator is
\[
y_{11} \otimes y_{11} = \left[ \begin{array}{cccc}
               a(1,1) \cdot a(1,1)  \quad & a(1,1) \cdot a(1,2) \quad & 0 & 0 \\
               a(1,2)  \cdot a(1,1) \quad & a(1,2)  \cdot a(1,2) \quad & 0 & 0 \\
               0 & 0 & 0 & 0 \\
               0 & 0 & 0 & 0 \end{array} \right]  
               \]
We can directly compute and verify that$\colon$
\[   \sum_{k=1}^{b} \sum_{w = 1}^{m} y_{kw} \otimes y_{kw} = I, \mbox{ the identity matrix.} \]
               
 \end{example}     
 
 \begin{rem}\label{rem:AstarA}
 The vectors $y_{kw} \in \mathbb{R}^{nb}$  may seem cumbersome at first but they enable us to write the matrix $A^{\ast}A$ in a manageable form.
 The matrix $A$ is constructed from $b$ blocks of $D_{j} W_{n}^{m}$ and so the matrix has the form
 \[
A = \left[ \begin{array}{ccccc}
               D_1 W_{n}^{m} \ | & D_2 W_{n}^{m} \ | & D_3 W_{n}^{m} \ | & D_4 W_{n}^{m} \ | \ \ldots \ldots  \ | & D_{b} W_{n}^{m} \end{array} \right]  
   \]
which means that when $b = 3$, the matrix $A^{\ast} A$ has the form,
\[  \left[ \begin{array}{ccc}
             \left(W_{n}^{m}\right)^{\ast}  & 0 & 0   \\
             0 & \left(W_{n}^{m}\right)^{\ast} & 0 \\
             0 & 0 & \left(W_{n}^{m}\right)^{\ast}
             \end{array}  \right] 
     \left[ \begin{array}{ccc}
             D_{1}^{\ast} D_{1} & D_{1}^{\ast} D_{2}  & D_{1}^{\ast} D_{3} \\
             D_{2}^{\ast} D_{1} & D_{2}^{\ast} D_{2}  & D_{2}^{\ast} D_{3}\\
             D_{3}^{\ast} D_{1} & D_{3}^{\ast} D_{2} & D_{3}^{\ast} D_{3}
             \end{array}  \right] 
     \left[ \begin{array}{ccc}
             W_{n}^{m}  & 0 & 0   \\
             0 & W_{n}^{m} & 0 \\
             0 & 0 & W_{n}^{m} 
             \end{array}  \right].
\]
For clarity, we have written out the form of $A^{\ast}A$ when $b = 3$.  The pattern extends to the  general case with $b$ blocks.
The key observation is that we can now write
\[
A^{\ast}A = \sum_{k=1}^{b} \sum_{j=1}^{b} \sum_{w=1}^{m} \theta_{kw} \theta_{jw} \ y_{kw} \otimes y_{jw} .
\]
This expression for $A^{\ast}A$ plays a crucial role in the proof of Theorem \ref{MainTheorem1}. \\

To show that a quantity $P$ is bounded by some constant, it is enough, as the next lemma tells us, to show that $P$ is bounded by some constant
multiplied by $(2 + \sqrt{P+1})$.

\begin{lem}\label{lemma:Beta}
Fix a constant $c_1 \leq 1$.  If $P > 0$ and
\[ P \leq c_1 \left( 2 + \sqrt{P + 1} \right),\]
then \[P < 5c_1.\]

\begin{proof}
Let $x = (P + 1)^{1/2}$ and note that $x$ is an increasing function of $P$.
The hypothesis of the lemma becomes
\[  x^2 - 1 \leq c_1 (2 + x)   \] which implies that \[x^2 - c_1 x -  (2 c_1  + 1)  \leq 0.  \]
The polynomial on the left is strictly increasing when $x \geq c_1 / 2.$
Since $\alpha \leq 1$ and $x \geq 1$ for $P \geq 0$, it is strictly increasing over the entire
domain of interest, thus
\[   x \leq \frac{ c_1 + \sqrt{ (c_1)^2 + 4(2 c_1  + 1)  }  }{2    }. \]
By substituting $(P + 1)^{1/2}$ back in for $x$, this means
\[  P + 1 \leq \frac{(c_1)^2}{4} + \frac{ c_1    \sqrt{ (c_1)^2 + 4(2 c_1  + 1)  }   }{  2    }  +  \frac{   \sqrt{ (c_1)^2 + 4(2 c_1  + 1)  }        }{ 4 }.    \]
Since $c_1 < 1$, this implies that $P < 5 c_1$.

\end{proof}

\end{lem}

\end{rem}


\section{Proof of RIP in expectation (Theorem \ref{MainTheorem1}) }
The rank-one operators $y_{kw} \otimes y_{kw}$ are constructed so that  
\begin{equation}\label{eqn:orthogonalVectors}
\sum_{k=1}^{b} \sum_{w = 1}^{m} y_{kw} \otimes y_{kw} = I.
\end{equation}
As explained in Remark (\ref{rem:AstarA}) from the last section, we have
\begin{equation}\label{eqn:AstarA}
A^{\ast}A = \sum_{k=1}^{b} \sum_{j=1}^{b} \sum_{w=1}^{m} \theta_{kw} \theta_{jw} \ y_{kw} \otimes y_{jw}. 
\end{equation}
The proof of the theorem proceeds by breaking up $I - A^{\ast}A$ into four different parts, then bounding the expected norm of each part separately.
By combining equations (\ref{eqn:orthogonalVectors}) and (\ref{eqn:AstarA}), we see that
\begin{equation}\label{eqn:fundamental}
I - A^{\ast}A = \sum_{k=1}^{b} \sum_{w=1}^{m} ( 1 - |\theta_{kw}|^2 ) \ y_{kw} \otimes y_{kw}  + \sum_{j \neq k} \sum_{w=1}^{m} \theta_{kw} \theta_{jw} \ y_{kw} \otimes y_{jw}. 
\end{equation}
For the two sums on the right hand side of (\ref{eqn:fundamental}), we will bound the expected norm of each sum separately.
Define two random quantities $Q_1$ and $Q_{1}'$ by
\begin{equation}\label{eqn:defineQ1}
Q_1 = \sum_{k=1}^{b} \sum_{w=1}^{m} (1 - |\theta_{kw}|^2) \ y_{kw} \otimes y_{kw}
\end{equation}
and
\begin{equation*}
Q_1' = \sum_{k=1}^{b} \sum_{w=1}^{m} (1 - |\theta_{kw}'|^2) \ y_{kw} \otimes y_{kw}
\end{equation*}
where $\{\theta_{kw}' \}$ is an independent copy of $\{\theta_{kw}\}$.  This implies that $Q_1$ has the same probability distribution as $Q_{1}'$.  To bound the expected norm of $Q_1$,
\begin{align*}
E \| Q_1 \|_{\Gamma} & = E \| Q_1 - E(Q_{1}') \|_{\Gamma} \\ 
& = E \| E [ Q_1 - Q_{1}' \ | Q_1 ] \ \|_{\Gamma} \\
& \leq E [ E \|  Q_1 - Q_{1}' \|_{\Gamma} \  | \ Q_1 ] \\
& = E (\|  Q_1 - Q_{1}' \|_{\Gamma} ).
\end{align*}
In the above equations, the first equality holds because $Q_{1}'$ has mean zero.  The second equality holds by the independence of $Q_1$ and $Q_{1}'$.  The inequality in the third line is true by Jensen's inequality.  Let
\begin{equation}\label{eqn:defineY} 
Y = Q_1 - Q_{1}' = \sum_{k=1}^{b} \sum_{w=1}^{m} ( | \theta_{kw}' |^2 - | \theta_{kw} |^2 ) y_{kw} \otimes y_{kw}.
\end{equation} 
We randomize this sum. The random variable $Y$ has the same probability distribution as
\begin{equation}\label{eqn:defineYprime}
Y' = \colon \sum_{k=1}^{b} \sum_{w=1}^{m} \epsilon_{kw} ( | \theta_{kw}' |^2 - | \theta_{kw} |^2 ) y_{kw} \otimes y_{kw}
\end{equation}
where $\{\epsilon_{kw}\}$ are independent, identically distributed Bernoulli random variables.
\begin{align*}
E \| Y \|_{\Gamma} &= E \| Y' \|_{\Gamma} \\
& = E \left[E \left( \| Y' \|_{\Gamma} \quad |  \{  \theta_{kw} \}, \{  \theta_{kw}' \} \right) \right].
\end{align*}
Let $x_{kw} = \left(   | \theta_{kw}' |^2 - | \theta_{kw} |^2 \right)^{1/2} y_{kw} $ in order to apply the lemma of Rudelson and Vershynin.  To see that each $x_{kw}$ is bounded, we note that
\begin{equation*}
B \geq  \max_{k, w} \left(   | \theta_{kw}' |^2 - | \theta_{kw} |^2 \right)^{1/2},
\end{equation*}
and so
 \[ \|  x_{kw} \|_{\infty} \leq \max_{k, w}  \left(   | \theta_{kw}' |^2 - | \theta_{kw} |^2 \right)^{1/2} \cdot \| y_{kw} \|_{\infty} \leq \frac{B}{ \sqrt{m}}. \]
 With the $\{  \theta_{kw} \}, \{ \theta_{kw}'  \}$ fixed, and with $K = B/ \sqrt{m}$, we apply Lemma (\ref{lem:RudelsonVershynin}) to obtain
 \begin{align*}  
 E \left[  \  \| Y' \|_{\Gamma}  \  |  \{  \theta_{kw} \}, \{ \theta_{kw}'  \}  \right]  \leq  \sqrt{ \frac{ C \cdot s \cdot L}{m} } \cdot B \cdot \left\| \sum_{k=1}^{b} \sum_{w = 1}^{m} \left(   | \theta_{kw}' |^2 - | \theta_{kw} |^2 \right) \cdot y_{kw} \otimes y_{kw}  \right\|_{\Gamma}^{1/2}
 \end{align*}
 where $L  \equiv \log^2 (s) \cdot \log(nb) \cdot \log(mb)$.  To remove the conditioning, we apply Cauchy-Schwarz inequality and the law of double expectation,
 \begin{equation}\label{eqn:YNorm}
  E \| Y \|_{\Gamma} \leq \sqrt{  \frac{C \cdot s \cdot L  }{ m }   }  \cdot B \left( \ E \left\| \sum_{k=1}^{b} \sum_{w = 1}^{m} \left(   | \theta_{kw}' |^2 - | \theta_{kw} |^2 \right) \cdot y_{kw} \otimes y_{kw}  \right\|_{\Gamma} \ \right)^{1/2}. 
  \end{equation}
 By using the triangle inequality,
 \begin{equation*}
 E \left\| \sum_{k=1}^{b} \sum_{w = 1}^{m} \left(   | \theta_{kw}' |^2 - | \theta_{kw} |^2 \right) \cdot y_{kw} \otimes y_{kw}  \right\|_{\Gamma}  \leq 2 E \left\|    \sum_{k=1}^{b} \sum_{w=1}^{m} | \theta_{kw}|^2 \cdot  y_{kw} \otimes y_{kw}  \right\|_{\Gamma}
 \end{equation*}
 and so the bound in (\ref{eqn:YNorm}) becomes
 \[  E \| Y \|_{\Gamma} \leq \sqrt{ \frac{C \cdot s \cdot L  }{m }  } \cdot \left(  E \left\|   \sum_{k=1}^{b} \sum_{w=1}^{m} | \theta_{kw}|^2 \cdot  y_{kw} \otimes y_{kw}  \right\|_{\Gamma} \right)^{1/2}.   \]
 Since $\sum_{k=1}^{b} \sum_{w=1}^{m} y_{kw} \otimes y_{kw} = I$ and since $ E \| I \|_{\Gamma} = 1$, we have
 \begin{align*}
 E \| Y \|_{\Gamma} & \leq \sqrt{ \frac{C  \cdot s \cdot L }{m}   } \cdot \left(  E \left\|   \sum_{k=1}^{b} \sum_{w=1}^{m} \left( 1 - | \theta_{kw}| \right)^2 \cdot  y_{kw} \otimes y_{kw}  \right\|_{\Gamma} + 1 \right)^{1/2} \\
 & = \sqrt{ \frac{C \cdot s \cdot L }{m}   } \cdot \left(   E \| Q_1 \|_{\Gamma}  + 1  \right)^{1/2}  \\
 & \leq \sqrt{ \frac{C \cdot s \cdot L }{m}   } \cdot \left(   E \| Y \|_{\Gamma}  + 1  \right)^{1/2}.
 \end{align*}
Solutions to the equation $E \leq \alpha (E + 1)^{1/2}$ satisfy $E \leq 2 \alpha$, where $\alpha \leq 1$. \\ Hence,
 the above inequalities show that
 there exist constants $C_{10}, C_{11}$ such that if $m \geq C_{10} \cdot s \cdot L $, then
 \begin{equation}\label{eqn:Q1Bound}
  E \| Q_1 \|_{\Gamma} \leq E \| Y \|_{\Gamma} \leq \sqrt{  \frac{C \cdot s \cdot L  }{m }.   }   
 \end{equation}
 We have now obtained a bound on the expected norm of the first sum in equation (\ref{eqn:fundamental}).  To control the norm of the second sum, we next define
 \begin{equation} Q_2 = \sum_{j \neq k} \sum_{w = 1}^{m} \theta_{kw}  \ \theta_{jw} \cdot y_{kw} \otimes y_{jw}   \end{equation}
 and we will apply decoupling inequality. Let
 \begin{equation}\label{eqn:defineQ2}
 Q_{2}' = \sum_{j \neq k} \sum_{w = 1}^{m} \theta_{kw} \ \theta_{jw}' \cdot y_{kw} \otimes y_{jw}  
 \end{equation}
 where $\{  \theta_{kw}' \}$ is an independent sequence with the same distribution as $\{\theta_{kw} \}$. Then 
 \[   E \| Q_2 \|_{\Gamma} \leq C_{12} \cdot E \| Q_{2}' \|_{\Gamma}. \]
 We will break up $Q_{2}'$ into two terms and control the norm of each one separately. 
 
 \begin{equation}\label{eqn:Q2prime}
 Q_{2}' = \sum_{j=1}^{b} \sum_{k=1}^{b} \sum_{w = 1}^{m} \theta_{kw} \ \theta_{jw}' \cdot y_{kw} \otimes y_{jw}  - \sum_{k=1}^{b} \sum_{w = 1}^{m} \theta_{kw} \ \theta_{jw}' \cdot y_{kw} \otimes y_{jw}.
 \end{equation}
 Denote the first term on the right by $Q_3$ and the second term on the right by $Q_4$.  To bound $\| Q_4 \|_{\Gamma}$, note that the random
 quantity $Q_4$ has the same distribution as
 \begin{equation}
 Q_{4}' = \sum_{k =1}^{b} \sum_{w = 1}^{m} \epsilon_{kw} \cdot u_{kw} \otimes v_{kw},
 \end{equation}
 where $u_{kw}$ and $v_{kw}$ are defined by
 \begin{equation}\label{eqn:Define_ukw}  u_{kw} = | \theta_{kw} | \cdot y_{kw}, \quad v_{kw} = | \theta_{kw}' | \cdot y_{kw}        \end{equation}
 and $\{ \epsilon_{kw} \}$ is an independent Bernoulli sequence.  Since $\max\{\theta_{kw}, \theta_{kw}'\} \leq B$, we have
 \[  \| u_{kw} \|_{\infty} \leq \frac{B}{ \sqrt{m} } \] and \[  \| v_{kw} \|_{\infty} \leq \frac{B}{ \sqrt{m}  } . \]
 Apply Lemma (\ref{lemma:Extension}) with $\{ \theta_{kw}, \theta_{kw}' \}$ fixed,
 \begin{align*}
 & E \left[   \| Q_{4}' \|_{\Gamma} \ | \{ \theta_{kw} \}, \{   \theta_{kw}' \}      \right] \\
 & \leq \sqrt{   \frac{C \cdot s \cdot L}{m} } \cdot B \cdot \left(  \left \|  \sum_{k=1}^{b} \sum_{w = 1}^{m}  | \theta_{kw} |^2 \cdot y_{kw} \otimes y_{kw}   \right\|_{\Gamma}^{1/2}  + \left \|  \sum_{k=1}^{b} \sum_{w = 1}^{m}  | \theta_{kw}' |^2 \cdot y_{kw} \otimes y_{kw}   \right\|_{\Gamma}^{1/2}      \right).
 \end{align*}
 Then, we use the law of double expectation and the Cauchy-Schwarz inequality, as in (\ref{eqn:YNorm}), to remove the conditioning$\colon$
{\small{ \begin{align*}
  E  \left[   \| Q_{4}' \|_{\Gamma} \right] & \leq  \sqrt{   \frac{C \cdot s \cdot L}{m} } \cdot B \\ & \cdot \left( E \left( \left[  \left \|  \sum_{k=1}^{b} \sum_{w = 1}^{m}  | \theta_{kw} |^2 \cdot y_{kw} \otimes y_{kw}   \right\|_{\Gamma}^{1/2}   + \left \|  \sum_{k=1}^{b}  \sum_{w = 1}^{m}  | \theta_{kw}' |^2 \cdot y_{kw} \otimes y_{kw}   \right\|_{\Gamma}^{1/2}      \right]^{2} \ \right) \right)^{1/2}.
 \end{align*} }}
 The two sequences of random variables $\{ \theta_{kw} \}$ and $\{ \theta_{kw}' \}$ are identically distributed, so using Jensen inequality, we get
 \begin{equation*}
  E \left[   \| Q_{4}' \|_{\Gamma} \right]  \leq \sqrt{   \frac{C \cdot s \cdot L}{m} } \cdot B \cdot \left(  E \left\|  \sum_{k=1}^{b} \sum_{w = 1}^{m}   | \theta_{kw} |^2 \cdot y_{kw} \otimes y_{kw}  \right\|_{\Gamma}  \right)^{1/2}.
 \end{equation*}
 To bound the expected value on the right-hand side, we note that
 \begin{align*}
& E \left(   \left\|     \sum_{k=1}^{b} \sum_{w = 1}^{m}   | \theta_{kw} |^2 \cdot y_{kw} \otimes y_{kw}     \right\|_{\Gamma}^{1/2}      \right) \\
& \leq \left(   E    \left\|     \sum_{k=1}^{b} \sum_{w = 1}^{m}   \left( 1 - | \theta_{kw} |^2 \right) \cdot y_{kw} \otimes y_{kw}     \right\|_{\Gamma} + 1    \right)^{1/2} \\
& = \left(  E \|  Q_1 \|_{\Gamma} + 1  \right)^{1/2}.
 \end{align*}
Recall that from equation $(\ref{eqn:Q1Bound})$,  if $m \geq C \cdot s \cdot L$, then $ E \| Q_1 \|_{\Gamma}$ is bounded.
 The random variables $\{ \theta_{kw} \}$ are bounded by the constant $B$, so we can conclude there exist constants $C_{13}, C_{14}$ such that
 if $m \geq C_{13} \cdot s \cdot L ,$ then
 \begin{equation}\label{eqn:Q4Bound}
 E \|  Q_4 \|_{\Gamma} \leq C_{14} \cdot \sqrt{  \frac{ s \cdot L }{ m  }     }.
 \end{equation}
   Recall that the right side of ($\ref{eqn:Q2prime}$) has two terms, $Q_3$ and $Q_4$. It remains to bound the expected norm of the other term,
 \[ Q_3 = \sum_{j = 1}^{b} \sum_{k = 1}^{b} \sum_{w = 1}^{m} \theta_{kw} \ \theta_{jw}' \  \cdot y_{kw} \otimes y_{kw}. \]
 To bound $E \| Q_3 \|_{\Gamma}$, note that $Q_3$ has the same probability distribution as
 \begin{align*}
 Q_{3}' & = \sum_{j = 1}^{b} \sum_{k = 1}^{b} \sum_{w = 1}^{m} \epsilon_{w} \  \theta_{kw} \ \theta_{jw}' \  \cdot y_{kw} \otimes y_{kw} \\
& = \sum_{w = 1}^{m} \epsilon_{w}    \left(  \sum_{k=1}^{b} \theta_{kw} y_{kw}        \right)   \otimes  \left(   \sum_{j = 1}^{b} \theta_{jw}'  y_{jw}      \right) \\
& = \sum_{w = 1}^{m} \epsilon_{w} \ u_{w} \otimes {v_w},
 \end{align*}
 where $u_{w}$ and $v_{w}$ are defined by
 \begin{equation}\label{eqn:Def_u_w} u_w = \sum_{k = 1}^{b} \theta_{kw} \ y_{kw} \ \mbox{ and } \ v_{w} = \sum_{j=1}^{b} \theta_{jw}' \ y_{jw}. \end{equation}
 The $y_{kw}$ have disjoint support for different values of $k$, so that
 \[ \| u_{w} \|_{\infty} \leq \frac{B}{\sqrt{m}}, \quad \| v_{w} \|_{\infty} \leq \frac{B}{\sqrt{m}}. \]
 Also note that
 \[ \sum_{w = 1}^{m} u_{w} \otimes v_{w} = \sum_{w=1}^{m} \sum_{k=1}^{b} \sum_{j=1}^{b} \theta_{kw} \ \theta_{jw} \ y_{kw} \otimes y_{jw}, \]
 and so by comparing equation (\ref{eqn:AstarA}) with the above expression, we see that \[ \sum_{w = 1}^{m} u_{w} \otimes u_{w} \] and \[ \sum_{w = 1}^{m} v_{w} \otimes v_{w} \] are independent copies of $A^{\ast}A$.
 By Lemma (\ref{lemma:Extension}) and Cauchy-Schwarz inequality, 
 \begin{align*}
 E \| Q_3 \|_{\Gamma} = E \| Q_{3}' \|_{\Gamma} & \leq \sqrt{    \frac{C \cdot s \cdot L'  }{m}  } \cdot B \cdot \left( E \|  A^{\ast}A \|_{\Gamma}   \right)^{1/2} \\
 & \leq C_{15} \sqrt{    \frac{C \cdot s \cdot L'  }{m}  } \cdot \left(    E \|  I + A^{\ast}A             \|_{\Gamma} + 1          \right)^{1/2}
 \end{align*}
 where we have written $L'  = \log^{2}(s) \cdot \log(m) \cdot \log(nb) $.\\
 Since $L' < L(s, n, m, b) \equiv \log^{2}{s} \cdot \log(mb) \cdot \log(nb)$, we can conclude that 
 \begin{equation*}
 E \| Q_3 \|_{\Gamma} \leq C_{15} \sqrt{\frac{s \cdot L}{m}} \cdot \left(    E \|  I + A^{\ast}A             \|_{\Gamma} + 1          \right)^{1/2}.
 \end{equation*}
 
 To recapitulate, we have shown that   
 \begin{align*}
 E \| I - A^{\ast}A  \|_{\Gamma} & \leq E \| Q_1  \|_{\Gamma} + E \| Q_2 \|_{\Gamma} \\
& \leq E \| Q_1  \|_{\Gamma} + C_{12} \cdot E \| Q_{2}' \|_{\Gamma} \\
& \leq E \| Q_1  \|_{\Gamma} + C_{12} \left( C_{15} E \| Q_{3}' \|_{\Gamma}  + C_{14} E \| Q_{4} \|_{\Gamma}       \right).
 \end{align*}
 
 When $m \geq \max(C_{10}, C_{13}) \cdot s \cdot L $, we have established that
 \begin{align*}
E \| Q_3  \|_{\Gamma} & \leq C_{15} \cdot \sqrt{   \frac{ s \cdot L   }{  m     }                  }  \cdot \left(   E \left\| I - A^{\ast}A \right\|_{\Gamma} + 1     \right)^{1/2} \\
E \| Q_4  \|_{\Gamma} & \leq C_{14} \cdot \sqrt{   \frac{ s \cdot L   }{  m     }                  } \\
E \| Q_1  \|_{\Gamma} & \leq C_{11} \cdot \sqrt{   \frac{ s \cdot L   }{  m     }                  }
 \end{align*}
Therefore, in conclusion, we have proven that
\[ E \left\| I - A^{\ast}A   \right\|_{\Gamma} \leq C \cdot    \sqrt{   \frac{ s \cdot L      }{  m     }                  } \cdot \left(  2 + \sqrt{  E \left\| I - A^{\ast}A \right\|_{\Gamma} + 1   }        \right). \]
 
 By Lemma (\ref{lemma:Beta}), with $P = E \| I - A^{\ast}A \|_{\Gamma}$ and $ c_1 = \sqrt{ \frac{C \cdot s \cdot L }{m }  } $,
  there exits constant $C_6$ such that if $m \geq C_6 \cdot s \cdot L $, then we have
 \[    E \left\| I - A^{\ast}A   \right\|_{\Gamma} \leq \sqrt{   \frac{ C \cdot s \cdot L    }{  m     }                  }.  \]
 Recall that by definition, $L  = \log^{2}{s} \cdot \log(mb) \cdot \log(nb)$.
 This proves that the assertion (\ref{eqn:MainThmNormBound}) is true.
 It remains to prove that equation (\ref{eqn:MainThmNormBound}) implies that 
 for any $0 < \delta_{s} \leq 1$, we have
\[   E\| I - A^{\ast}A \|_{\Gamma} \leq \delta_{s}  \]
 provided that $m \geq C_{7} \cdot \delta_{s}^{-2} \cdot s \cdot \log^{4}(nb)$ and $m \leq n b$.\\

If $m \leq n b$, then $\log(mb) = \log(m) + \log(b) \leq 2 \log(nb)$.  Hence, for $s \leq nb$,
\begin{equation}\label{eqn:Log_mb_inequality}
 \sqrt{\frac{C_5 \cdot s \cdot \log^{2}(s)  \cdot \log(mb) \cdot \log(nb)}{m}}  \leq \sqrt{\frac{4 C_{5} \cdot s \log^{4}(nb)}{m}}. 
\end{equation}
The right-hand side is bounded by $\delta_s$ when $m \geq C_{7} \cdot \delta_{s}^{-2} \cdot s \cdot \log^{4}(nb)$.
 
 This concludes the proof Theorem \ref{MainTheorem1}.



 \section{Proof of the tail bound (Theorem \ref{Thm:TailBound})}
Recall that $I - A^{\ast}A = Q_1 + Q_2$, where the expressions for $Q_1$ and $Q_2$ are given in equations (\ref{eqn:defineQ1}) and (\ref{eqn:defineQ2}).  To obtain a bound for $P( \| I - A^{\ast}A \|_{\Gamma} > \delta),$ we will first 
 find a bound for $P( \| Q_1\|_{\Gamma} >  \frac{\delta}{2}).$  Recall that $Y = Q_1 - Q_{1}'$ from (\ref{eqn:defineY}), and $Y'$ has the same probability distribution as $Y$ from (\ref{eqn:defineYprime}).  For any $\beta > 0$ and for any $\lambda > 0$, 
 \begin{equation}\label{eqn:beta_lambda_Y}
  P( \| Q_{1}' \|_{\Gamma} < \beta) P( \| Q_1 \|_{\Gamma} > \beta + \lambda ) \leq P( \| Y \|_{\Gamma} > \lambda). 
 \end{equation}
 This equation holds because if $\| Q_1\|_{\Gamma} > \beta + \lambda$ and if $\| Q_{1}' \|_{\Gamma} < \beta$, then
 \[\beta + \lambda < \| Q_1 \|_{\Gamma} \leq \| Q_{1}' \|_{\Gamma} + \| Y \|_{\Gamma} < \beta + \| Y \|_{\Gamma}, \]
 and so $\| Y \|_{\Gamma}$ must be greater than $\lambda$.
 Note that the median of a positive random variable is never bigger than two times the mean, therefore
 \[ P( \| Q_{1}' \|_{\Gamma} \leq 2 E \| Q_{1}' \|_{\Gamma} ) \geq \frac{1}{2}. \]
 We can choose $\beta = 2 E \| Q_{1}' \|_{\Gamma}$ so that (\ref{eqn:beta_lambda_Y}) becomes
 \[ P( \| Q_{1}' \|_{\Gamma} < 2 E \| Q_{1}' \|_{\Gamma}) P( \| Q_1 \|_{\Gamma} > 2 E \| Q_{1}' \|_{\Gamma} + \lambda ) \leq P( \| Y \|_{\Gamma} > \lambda). \]
 Since $E( \| Q_1 \|_{\Gamma}) = E \| Q_{1}' \|_{\Gamma}$, we obtain
 \[ P( \| Q_1 \|_{\Gamma} > 2 E \| Q_1 \|_{\Gamma} + \lambda ) \leq 2 P( \| Y \|_{\Gamma} > \lambda). \]
 Let $ V_{k,w}  = | |\theta_k(w)'|^2 - | \theta_k(w)|^2 | \cdot y_{k,w} \otimes y_{k,w}$.  Then $\| V_{k,w} \| \leq K \cdot \frac{s}{m}$, where we define $K = B^2$. 
 By the Proposition \ref{Prop:ConcentrationIneq}, the concentration inequality gives us
 \[ P \left(  \| Y' \|_{\Gamma} > C \left[  u E  \| Y' \|_{\Gamma}  + \frac{Ks}{m} t \right] \right) \leq e^{-u^2} + e^{-t}. \]
 From (\ref{eqn:Q1Bound}), we have the bound
 \[ E \| Y' \|_{\Gamma} \leq \sqrt{ C \cdot \frac{s \cdot KL}{m}} \quad \mbox{ where } m \geq C \cdot s \cdot L \cdot \log(mb). \]
 Combining the last two inequalities, we see that
 \[ P \left(  \| Y' \|_{\Gamma} > C \left[  u \sqrt{\frac{sKL}{m}}  + \frac{Ks}{m} t \right] \right) \leq e^{-u^2} + e^{-t}. \]
 Fix a constant $\alpha$, where $0 < \alpha < 1/10$.
 If we pick $t = \log(1/\alpha)$ and $u = \sqrt{\log(1/ \alpha) }$, then the above equation becomes
  \[ P \left(  \| Y' \|_{\Gamma} > C \left[ \sqrt{\frac{ sKL \log(1/ \alpha)}{ m} } + \frac{sK}{m} \log(1/ \alpha) \right)  \right] \leq 2 \alpha. \]
  That means for some constant $\lambda$, we have
 \[ P \left(  \| Y' \|_{\Gamma} > \lambda \right) \leq 2 \alpha.\]
 With this constant $\lambda$, we use the bound for $E \| Q_1 \|_{\Gamma}$ in (\ref{eqn:Q1Bound}) to conclude that
 \[ 2 E \| Q_1 \|_{\Gamma} + \lambda \leq C \left(  \sqrt{ \frac{s L \log(mb) }{m}} + \sqrt{\frac{s KL \log(1/ \alpha) }{m}} + \log(1/ \alpha) \frac{sK}{m} \right). \]
Inside the bracket on the right side, when we choose $m$ so that all three terms are less than 1, the middle term will dominate the other two terms.  That means there exists a constant $C_{16}$ such that if
\[ m \geq C_{16} \cdot \delta_{s}^{-2} \cdot s \cdot K \cdot L  \cdot \log(1/ \alpha) \]
then \[ 2 E \| Q_1 \|_{\Gamma} + \lambda \leq \frac{\delta_s}{2}. \]
Combining the above inequalities, we obtain
\[ P\left( \| Q_1 \|_{\Gamma} > \frac{\delta_s}{2}  \right) \leq 4 \alpha. \]
Next, we want to obtain a bound for $\| Q_2 \|_{\Gamma}$.  We saw that from (\ref{eqn:Q2prime}), $Q_{2}' = Q_3 + Q_4$, and so to obtain a bound for  $\| Q_2 \|_{\Gamma}$, we will proceed to find bounds for $ \| Q_3 \|_{\Gamma}$ and $\| Q_4 \|_{\Gamma}$.\\

Recall that  $Q_4$ has the same probability distribution as $Q_{4}'$ and we showed that \[E \| Q_{4}' \| \leq \sqrt{\frac{C \cdot s \cdot L \cdot \log(mb)}{m}}.\]  With the same  $u_{k,w}$ and $v_{k,w}$ defined in (\ref{eqn:Define_ukw}), we have $\| u_{k,w} \otimes v_{k,w} \| \leq K \frac{s}{m}.$ \\
Apply the concentration inequality (Proposition \ref{Prop:ConcentrationIneq}) with $t = \log(C/ \alpha)$ and $u = \sqrt{\log(C/ \alpha)}$, we obtain
\[ P\left( \| Q_{4}' \|_{\Gamma} > C \left[ \sqrt{\frac{s \cdot KL \cdot \log(C/ \alpha)}{m}} + \frac{s \cdot K \cdot \log(C/ \alpha)}{m} \right]  \right) \leq \frac{2 \alpha}{C}. \]
Inside the probability bracket, when we choose $m$ so that both terms on the right side are less than 1, the first term will dominate the second term.  That means there exists a constant $C_{18}$ such that if
\[ m \geq C_{18} \cdot \delta_{s}^{-2} \cdot s \cdot K \cdot L \cdot \log(1 / \alpha) \]
then \[ P\left(   \| Q_{4}' \|_{\Gamma} > \delta_{4} \cdot C    \right) \leq \frac{2 \alpha}{C_{17}}. \]
Since $Q_4$ and $Q_{4}'$ have the same probability distribution, we finally arrive at 
\[  P\left( \| Q_4 \|_{\Gamma} > \delta_{4} \cdot C \right) \leq \frac{2 \alpha}{C_{17}}. \]
To obtain a bound for $\| Q_3 \|$, we will follow similar strategy.  Recall that $Q_3$ has the same probability distribution as $Q_{3}'$ and we showed that
\[  E \| Q_3 \|_{\Gamma} \leq E \left(\|  I - A^{\ast}A \|_{\Gamma} + 1   \right)^{1/2} \sqrt{ \frac{C s L \log(mb)}{m}  }. \]
With the same  $u_w$ and $v_w$ defined in (\ref{eqn:Def_u_w}), we have $\| u_w \otimes v_w \| \leq K \frac{s}{m}. $

By the concentration inequality (Proposition \ref{Prop:ConcentrationIneq}) with $t = \log(C/ \alpha)$ and $u = \sqrt{ \log(C/ \alpha) }$, we obtain
\[  P\left( \| Q_3 \|_{\Gamma} > C \left[ \sqrt{ \frac{s \cdot L \cdot \log(mb) \cdot \log(C/ \alpha) }{ m } } + \frac{s \cdot K \cdot \log(C/ \alpha)}{m} \right]    \right) \leq \frac{2 \alpha}{C_{17}}.  \]
That means there exists a constant $C_{19}$ such that if
\[m \geq C_{19} \cdot \delta_{s}^{-2} \cdot s \cdot KL \cdot \log(1/ \alpha) \]
then
\[ P\left( \| Q_3 \|_{\Gamma} > C \cdot \delta_{s} \right) \leq \frac{2 \alpha}{C_{17}}. \]
To summarize, we have shown for any $0 < \alpha < \frac{1}{10}$, there exists $C_{20}$ such that when
$m \geq C_{20} \delta_{s}^{-2} \cdot s \cdot K \cdot L \cdot \log(1/ \alpha)$, all the following inequalities hold:
\begin{equation}\label{eqn:Ineq1Thm2} P\left( \| Q_1 \|_{\Gamma} > \frac{\delta_s}{2}  \right) \leq 4 \alpha. \end{equation}
\begin{equation}\label{eqn:Ineq2Thm2}P\left( \| Q_3 \|_{\Gamma} > C \cdot \delta_{s} \right) \leq \frac{2 \alpha}{C_{17}}. \end{equation}
\begin{equation}\label{eqn:Ineq3Thm2}P\left( \| Q_4 \|_{\Gamma} > \delta_{4} \cdot C \right) \leq \frac{2 \alpha}{C_{17}}. \end{equation}
Finally, to find a tail bound on $\| I - A^{\ast}A \|_{\Gamma}$, we note that
\begin{align*}
& P( \| I - A^{\ast}A \|_{\Gamma} > \delta_s ) \\
& \leq P( \| Q_1 \|_{\Gamma} > \frac{\delta_s}{2} ) + P( \| Q_2 \|_{\Gamma} > \frac{\delta_s}{2} ) \\
& \leq P( \| Q_1 \|_{\Gamma} > \frac{\delta_s}{2} ) + C_{12} \cdot P( \| Q_{2}' \|_{\Gamma} > \frac{\delta_s}{2}C_{12} ) \\
& \leq P( \| Q_1 \|_{\Gamma} > \frac{\delta_s}{2} ) + C_{12} \cdot P( \| Q_{3} \|_{\Gamma} > \frac{\delta_s}{4}C_{12} ) + C_{12} \cdot P( \| Q_{4} \|_{\Gamma} > \frac{\delta_s}{4}C_{12} )
\end{align*}
Combining this observation with inequalities (\ref{eqn:Ineq1Thm2}), (\ref{eqn:Ineq2Thm2}), (\ref{eqn:Ineq3Thm2}), we can conclude that for any $0 < \alpha < \frac{1}{10}$, we have \[P( \| I - A^{\ast}A \| > \delta_{s}) \leq 8 \alpha \]
when $m \geq C_{20} \delta_{s}^{-2} \cdot s \cdot K \cdot L \cdot \log(1/ \alpha)$.
Recall that by definition, $L= \log^{2}(s) \log(mb) \log(nb)$. \\

 Note that $m \leq nb$, and hence equation (\ref{eqn:Log_mb_inequality}) remains valid.  To complete the proof, we replace $8 \alpha$ with $\epsilon$. This completes the proof of the theorem.\\

\bibliographystyle{plain}
\bibliography{2013RMbib}

\begin{thebibliography}{10}

\bibitem{BaraniukCevherDuarte2010}
Richard Baraniuk, Volkan Cevher, Marco~F. Duarte, and Chinmay Hegde.
\newblock Model-based compressive sensing.
\newblock {\em IEEE Trans. Inform. Theory}, 56(4):1982--2001, 2010.

\bibitem{BaraniukDavenportDeVoreWakin2008}
Richard Baraniuk, Mark Davenport, Ronald DeVore, and Michael Wakin.
\newblock A simple proof of the restricted isometry property for random
  matrices.
\newblock {\em Constr. Approx.}, 28(3):253--263, 2008.

\bibitem{BasriJacobs1997}
Ronen Basri and David~W. Jacobs.
\newblock Lambertian reflection and linear subspaces.
\newblock {\em IEEE Trans. Pattern Anal. Mach. Intell.}, 25(3):218--233, 2003.

\bibitem{HespanhaKriegman1997}
Peter Belhumeur, J.~Hespanha, and David Kriegman.
\newblock Eigenfaces versus {F}isherfaces: {R}ecognition using class specific
  linear projection.
\newblock {\em IEEE Trans. Pattern Anal. Mach. Intell.}, 19(7):711--720, 1997.

\bibitem{BlumensathDavies2009}
Thomas Blumensath and Mike Davies.
\newblock Iterative hard thresholding for compressed sensing.
\newblock {\em Appl. Comput. Harmon. Anal.}, 27(3):265--74, 2009.

\bibitem{BrucksteinDonohoElad2009}
Alfred Bruckstein, David~L. Donoho, and Michael Elad.
\newblock From sparse solutions of systems of equations to sparse modeling of
  signals and images.
\newblock {\em SIAM Rev.}, 51(1):34--81, 2009.

\bibitem{Candes2008}
Emmanuel Candes.
\newblock The restricted isometry property and its implications for compressed
  sensing.
\newblock {\em C. R. Math. Acad. Sci. Paris}, 346(9-10):589--592, 2008.

\bibitem{CandesRombergTao2006a}
Emmanuel Candes, Justin Romberg, and Terence Tao.
\newblock Robust uncertainty principles: exact signal reconstruction from
  highly incomplete frequency information.
\newblock {\em IEEE Trans. Inform. Theory}, 52(2):489--509, 2006.

\bibitem{CandesRombergTao2006b}
Emmanuel Candes, Justin Romberg, and Terence Tao.
\newblock Stable signal recovery from incomplete and inaccurate measurements.
\newblock {\em Comm. Pure Appl. Math.}, 59(8):1207--1223, 2006.

\bibitem{CohenDahmenDeVore2009}
Albert Cohen, Wolfgang Dahmen, and Ronald DeVore.
\newblock Compressed sensing and best k-term approximation.
\newblock {\em J. Amer. Math. Soc.}, 22(1):211--231, 2009.

\bibitem{Donoho2006}
David~L. Donoho.
\newblock Compressed sensing.
\newblock {\em IEEE Trans. Inform. Theory}, 52(4):1289--1306, 2006.

\bibitem{DonohoTanner2009}
David~L. Donoho and Jared Tanner.
\newblock Counting faces of randomly projected polytopes when the projection
  radically lowers dimension.
\newblock {\em J. Amer. Math. Soc.}, 22(1):1--53, 2009.

\bibitem{Foucart2010}
Simon Foucart.
\newblock Sparse recovery algorithms: sufficient conditions in terms of
  restricted isometry constants.
\newblock In {\em Approximation Theory XIII: San Antonio 2010}, volume~13 of
  {\em Springer Proceedings in {M}athematics}, pages 65--77. Springer, New
  York, 2012.

\bibitem{FoucartRauhut2013}
Simon Foucart and Holger Rauhut.
\newblock {\em A {M}athematical {I}ntroduction to {C}ompressive {S}ensing}.
\newblock Applied and Numerical Harmonic Analysis. Birkh\"auser, New York,
  2013.

\bibitem{BajwaNowak2007}
J.~Haupt, W.~U. Bajwa, G.~Raz, S.J. Wright, and R.~Nowak.
\newblock Toeplitz-structured compressed sensing matrices.
\newblock In {\em Proc. 14th IEEE/SP {W}orkshop {S}tat. {S}ignal {P}rocess.
  (SSP'07), Madison, WI, Aug 2007}, pages 294--298.

\bibitem{KrahmerWard2012}
Felix Krahmer and Rachael Ward.
\newblock New and improved {J}ohnson-{L}indenstrauss embeddings via the
  restricted isometry property.
\newblock {\em SIAM J. Math. Anal.}, 43(3):1269--1281, 2011.

\bibitem{Talagrand1991}
Michel Ledoux and Michel Talagrand.
\newblock {\em Probability in {B}anach spaces}.
\newblock A {S}eries of {M}odern {S}urveys in {M}athematics. Springer-Verlag,
  Berlin, 1991.

\bibitem{LiCong2015}
Kezhi Li and Shuang Cong.
\newblock State of the art and prospects of structured sensing matrices in
  compressed sensing.
\newblock {\em Front. Comput. Sci.}, 9(5):665--677, 2015.

\bibitem{MendelsonPajorJaegermann2008}
Shahar Mendelson, Alain Pajor, and Nicole Tomczak-Jaegermann.
\newblock Uniform uncertainty principle for {B}ernoulli and {S}ubgaussian
  ensembles.
\newblock {\em Constr. Approx.}, 28(3):277--289, 2008.

\bibitem{Rauhut2007}
Holger Rauhut.
\newblock Random sampling of sparse trigonometric polynomials.
\newblock {\em Appl. Comput. Harmon. Anal.}, 22(1):16--42, 2007.

\bibitem{Rauhut2010}
Holger Rauhut.
\newblock Compressive sensing and structured random matrices.
\newblock In {\em Theoretical foundations and numerical methods for sparse
  recovery}, Radon {S}er. {C}omput. {A}ppl. {M}ath., 9. Walter de {G}ruyter,
  Berlin, 2010.

\bibitem{RauhutWard2012}
Holger Rauhut and Rachael Ward.
\newblock Sparse {L}egendre expansions via $l_1$-minimization.
\newblock {\em J. Approx. Theory}, 164(5):517--533, 2012.

\bibitem{RudelsonVershynin2008}
Mark Rudelson and Roman Vershynin.
\newblock On sparse reconstruction from {F}ourier and {G}aussian measurements.
\newblock {\em Comm. Pure Appl. Math.}, 61(8):1025--1045, 2008.

\bibitem{Talagrand1996}
Michel Talagrand.
\newblock Majorizing measures: the generic chaining.
\newblock {\em Ann. Probab.}, 24(3):1049--1103, 1996.

\bibitem{Tropp2004}
Joel~A. Tropp.
\newblock Greed is good: algorithmic results for sparse approximation.
\newblock {\em IEEE Trans. Inform. Theory}, 50(10):2231--2242, 2004.

\bibitem{Tropp2006}
Joel~A. Tropp.
\newblock Just relax: convex programming methods for identifying sparse signals
  in noise.
\newblock {\em IEEE Trans. Inform. Theory}, 52(3):1030--1051, 2006.

\bibitem{TroppRombergBaraniuk2010}
Joel~A. Tropp, Jason~N. Laska, Marco~F. Duarte, Justin Romberg, and Richard
  Baraniuk.
\newblock Beyond {N}yquist: Efficient sampling of sparse bandlimited signals.
\newblock {\em IEEE Trans. Inform. Theory}, 56(1):540, 2010.

\bibitem{WrightYangSastryMa2009}
John Wright, Allen~Y. Yang, Arvind Ganesh, S.~Shankar Sastry, and Yi~Ma.
\newblock Robust face recognition via sparse representation.
\newblock {\em IEEE Trans. Pattern Anal. Mach. Intell.}, 31(2):210--227, 2009.

\end{thebibliography}

\end{document}